\documentclass[11]{amsart}

\usepackage{mathpazo}
\usepackage[mathscr]{eucal}
\usepackage{enumerate}
\usepackage[psamsfonts]{amssymb}
\usepackage{pdfsync}
\usepackage{graphicx,psfrag}
\usepackage[psamsfonts]{amssymb}
\usepackage{pdfsync}
\usepackage{amscd}
\usepackage{amsmath}
\usepackage{amsxtra}
\usepackage{mathrsfs}
\usepackage{stmaryrd}
\usepackage{slashed}
\usepackage{comment}
\usepackage{enumerate}
\usepackage{hyperref}

\usepackage[width=5.8in, height=8.5in, bottom=1.3in, margin=1.2in, centering]{geometry}

\usepackage{amsmath,amsfonts,amssymb}
\usepackage[all]{xy}
\usepackage{pb-diagram}
\usepackage{hyperref}

\usepackage{mathtools}

\usepackage{mathrsfs}

\theoremstyle{plain}
        \newtheorem{theorem}{Theorem}[section]
        \newtheorem{lemma}[theorem]{Lemma}
        \newtheorem{proposition}[theorem]{Proposition}
        \newtheorem{corollary}[theorem]{Corollary}
        
        \theoremstyle{definition}
        \newtheorem{definition}[theorem]{Definition}

\newcommand{\C}{{\mathbb C}}
\newcommand{\R}{{\mathbb R}}

\newcommand{\g}{\mathfrak{g}}

\newcommand{\sfG}{\mathsf{G}}

\newcommand{\rk}{\operatorname{rk}}

\newcommand{\calD}{\mathcal{D}}
\newcommand {\calA}{\mathcal{A}}
\newcommand{\calF}{\mathcal{F}}

\newcommand{\calU}{\mathcal{U}}

\newcommand{\calC}{\mathcal{C}}

\newcommand{\doublegroupoid}[4]{\xymatrix{#1  \ar@<2pt>[d] \ar@<-2pt>[d] & #2 \ar@<2pt>[l] 
\ar@<-2pt>[l] \ar@<2pt>[d] \ar@<-2pt>[d] \\ #3  & #4 \ar@<2pt>[l] \ar@<-2pt>[l]}}

\title{The index of geometric operators on Lie groupoids}

\author{M.J.~Pflaum, H. Posthuma,~\textrm{and} X.~Tang}
\address{\newline
Markus J. Pflaum, {\tt markus.pflaum@colorado.edu}\newline
         \indent {\rm Department of Mathematics, University of Colorado,
         Boulder, USA}
         \newline
        Hessel Posthuma, {\tt H.B.Posthuma@uva.nl}\newline
         \indent {\rm Korteweg-de Vries Institute for Mathematics,
        University of Amsterdam,
         The Netherlands} 
         \newline
        Xiang Tang, {\tt xtang@math.wustl.edu}   \newline
         \indent {\rm  Department of Mathematics, Washington University,
         St.~Louis, USA}}

\begin{document}
\maketitle
\begin{abstract}
We revisit the cohomological index theorem for elliptic elements in the universal enveloping algebra of a 
Lie groupoid previously proved by the authors. We prove a Thom isomorphism for Lie algebroids which 
enables us to rewrite the ``topological side'' of the index theorem. This results in index formulae for Lie 
groupoid analogues of the familiar geometric operators on manifolds such as the signature and Dirac 
operator expressed in terms of the usual characteristic classes in Lie algebroid cohomology. 
\end{abstract}
\section*{Introduction}
The Atiyah--Singer index theorem for elliptic operators on compact manifolds lays a fundamental bridge 
between analysis and topology: it gives a topological formula involving characteristic classes computing 
the index of an operator, defined analytically as the difference in dimension between the kernel and the 
cokernel. Well known specific examples of the index theorem include the Gauss--Bonnet--Chern
theorem and Hirzebruch's signature theorem. These are obtained by considering the geometrically 
most natural elliptic differential operators on a compact manifold equipped with a metric: the Euler 
operator and the signature operator. Assuming that the underlying manifold is {\em spin} the Dirac 
operator is another example of an elliptic differential operator completely determined by the underlying 
geometry. 

Ever since the first proof was announced by Atiyah and Singer, the index theorem has been generalized 
in a multitude of directions. We mention the following: 1) the index theorem for families was part of the 
original series of papers by Atiyah--Singer \cite{as}. 2) in \cite{cs}, the family index theorem was generalized by 
Connes--Skandalis to longitudinal operators on compact foliated manifolds on the level of $K$-theory. In 
\cite{connes}, Connes deduced from this the cohomological version of the index theorem for 
foliations. 

The index theorem of \cite{ppt}, recently proved by the authors, takes this chain of generalizations one
step further: it gives a cohomological formula for the index, analytically defined by means of noncommutative 
geometry, of invariant elliptic operators along the fibers of a Lie groupoid over a compact base manifold. The
original Atiyah--Singer index theorem is recovered by considering the so-called {\em pair groupoid}, 
whereas the family index theorem follows using the natural groupoid associated to a submersion of 
manifolds. Finally, the longitudinal index theorem for foliations is obtained by means of the holonomy groupoid of a 
foliation, as it was already used in the original derivation in \cite{cs}.

The step generalizing from foliations to arbitrary Lie groupoids, means that we can now consider certain 
{\em singular foliations}, namely those induced by the image of the anchor of an integrable Lie 
algebroid. One of the main points of this article is to argue that this class of singular foliations
 is in fact quite well-behaved: using the geometry of Lie algebroids, one can define the singular 
 version of the geometric differential operators mentioned above, and apply the index theorem to
 these operators. 
 
 One of the mysterious looking features of the index theorem of \cite{ppt} is that the topological side involves integration over the dual of the Lie algebroid of cohomology classes of the pull-back of the Lie algebroid. This is analogous to the cohomological form of the Atiyah--Singer index theorem as
 an integral over the cotangent bundle of the original manifold. In this case the Thom isomorphism 
 provides the correspondence to integral formulae over the manifold itself.
Therefore, to get useful formulae in our case, we prove a version of the Thom isomorphism for 
Lie algebroids. Applied to the geometric operators mentioned above, the Euler, signature and Dirac 
Operator, we obtain index formulae that are completely similar to the usual cohomological index formulae, except for the fact that the characteristic classes take values in Lie algebroid cohomology
instead of de Rham cohomology.

This paper is organized as follows. In \S \ref{index} we briefly review the set-up and the statement of the
index theorem of \cite{ppt}. After that, in \S \ref{Thom}, we prove a Lie algebroid version of the Thom isomorphism. Then we discuss the construction of the natural geometric operators on Lie groupoids in
\S \ref{geom}, and use the Thom isomorphism to calculate the index formulae for these operators.
 \vskip 2mm
 
\noindent{\bf Acknowledgments}:  Pflaum is partially supported by NSF grant DMS 1105670, and 
Tang is partially supported by NSF grant DMS 0900985 and NSA grant H96230-13-1-0209. 

\section{The index theorem for Lie groupoids}
\label{index}
\subsection{Preliminaries}
The setting of the index theorem of \cite{ppt} is that of Lie groupoids. Here we shall describe the 
most basic definitions that we need. We refer to \cite{mackenzie, mm} for more details on the general 
theory of Lie groupoids.
Throughout this paper, $\sfG\rightrightarrows M$ denotes a Lie groupoid over a base manifold $M$, 
which is assumed to be compact. We write source and target maps as $s,t:\sfG\to M$, respectively, and denote the unit map 
by $u :M \rightarrow \sfG$. An arrow $g\in\sfG$ with $s(g)=x$ and $t(g)=y$ is written as $x\stackrel{g}{\rightarrow} y$, so 
that the multiplication $g_1g_2$ is defined when $t(g_1)=s(g_2)$. In general, we write $\sfG_k$ for the 
space of $k$-tuples of composable arrows
\[
  \sfG_k:=\{(g_1,\ldots,g_k)\in\sfG^{\times k},~s(g_i)=t(g_{i+1}),~i=1,\ldots,k-1\}.
\]
By definition of a Lie groupoid, $s$ and $t$ are assumed to be smooth submersions, so that $\sfG_k$ 
has a smooth manifold structure. With this, the multiplication is a smooth map $\sfG_2\to\sfG.$

The {\em Lie algebroid} of the groupoid $\sfG\rightrightarrows M$ is defined as the vector bundle
$\ker(dt)|_M$ over $M$, and is denoted by $\g$. Its space of smooth sections $\Gamma^\infty(M;\g)$ 
carries a Lie algebra structure $[~,~]$, together with an anchor map $\rho:\g\to TM$ satisfying
\begin{align*}
\rho([X,Y])&=[\rho(X),\rho(Y)],\\
[X,fY]&=f[X,Y]+\rho(X)(f)\cdot Y, 
\end{align*}
for $X,Y\in \Gamma^\infty (M;\g)$, and $f \in \calC^\infty (M)$. 

A {\em representation} of $\sfG\rightrightarrows M$ is a vector bundle $E\to M$ equipped with a linear action of $\sfG$. 
This means that to each $g\in \sfG$ is assigned a $\lambda_g\in{\rm Hom}(E_{s(g)},E_{t(g)})$ satisfying the properties
\begin{align*}
\lambda_{u(x)}&={\rm id}_{E_{u(x)}},\quad\mbox{for all}~x\in M,\\
\lambda_{g_1}\lambda_{g_2}&=\lambda_{g_1g_2},\quad\mbox{for all}~(g_1,g_2)\in \sfG_2.
\end{align*}
When a vector bundle $E\to M$ carries a representation of $\sfG\rightrightarrows M$, we can set up 
a complex computing the {\em differentiable (or smooth) cohomology} as follows: Define $C^k_{\rm diff}(\sfG;E):=\Gamma(\sfG_k;\tau_k^*E)$, 
where $\tau_k:\sfG_k\to M$ is defined by $\tau(g_1,\ldots,g_k)=t(g_1)$. The differential 
$d:C^k_{\rm diff}(\sfG;E)\to C^{k+1}_{\rm diff}(\sfG;E)$ is defined for $k\geq 1$ as
\[
\begin{split}
d\varphi(g_1,&\ldots,g_{k+1}):=\lambda_{g_1}\varphi(g_2,\ldots,g_{k+1})\\
&+\sum_{i=1}^k(-1)^k\varphi(g_1,\ldots,g_ig_{i+1},\ldots,g_k)+(-1)^{k+1}\varphi(g_1,\ldots,g_k),
\end{split}
\]
and in degree $k=0$ we simply have
\[
d\varphi(g):=\lambda_g\varphi(s(g))-\varphi(t(g)).
\]
Its cohomology is denoted by $H^\bullet_{\rm diff}(\sfG;E)$. From the formula above it is clear that 
$H^0_{\rm diff}(\sfG;E)$ consists of smooth invariant sections of $E$.

All these concepts and definitions above are rather straightforward generalizations of the theory of Lie 
groups. A striking difference however is that Lie groupoids in general do not have an adjoint 
representation, nor does the tangent bundle $TM$ support a representation of $\sfG$. It therefore 
comes as a surprise that, as proved in \cite{elw}, a canonical representation of $\sfG$ is defined on the line bundle
\begin{equation}
\label{density-bundle}
L_\g=\bigwedge^{\rm top}\g\otimes\bigwedge^{\rm top}T^*M.
\end{equation}
Sections of this line bundle are called ``transversal densities'' by analogy with the theory of foliations: when $\calF\subset TM$ is an integrable sub bundle, it obviously defines a Lie algebroid with anchor given by the inclusion map. In this case the anchor has constant rank and the normal bundle $\nu_\calF:=TM\slash\calF$ is well defined. It follows that $L_\calF\cong\bigwedge^{\rm top}\nu_\calF^*$ in this case. And nonvanishing invariant sections of this bundle are exactly transversal densities, c.f.\ \cite{ms}.

\subsection{The index theorem}
In this section we explain the statement of the main index theorem of \cite{ppt} for Lie groupoids. 
The structure of this index theorem is the same as that of the original Atiyah--Singer index theorem:
it gives an equality between two complex numbers, one defined {\em analytically} and the other 
{\em topologically}. The proper context of the topological right hand side is that of characteristic classes 
in Lie algebroid  cohomology, the details of which will be explained in \S \ref{Thom}. We will therefore 
introduce the index theorem by first explaining the left hand side, the analytical index.

The right framework for this is that of noncommutative geometry \cite{connes}. A Lie groupoid 
$\sfG\rightrightarrows M$ defines a {\em smooth convolution algebra} 
\[
\calA_\sfG:=\Gamma_\textup{c}^\infty\left( \sfG,s^*{\bigwedge}^\textup{top} \g^* \right),
\]
equipped with the product
\[
(f_1*f_2)(g):=\int_{{t(h)=t(g)}} f_1(gh^{-1}) f_2(h) \, .
\]
It is perhaps not obvious that this multiplication is well-defined, however see \cite{ppt} for an 
explanation. In \cite[\S II.5]{connes}, Connes uses half-densities rather than full densities as above.
This is more convenient for the definition of a $*$-structure, important when one wants to 
consider the completion to a $C^*$-algebra. 


The idea of noncommutative geometry is to consider 
the convolution algebra $\calA_\sfG$ as a noncommutative ``ring of functions'' on a space.
As such, one of the most interesting invariants to compute is the cyclic homology of $\calA_\sfG$, the 
proper generalization of de Rham cohomology to the noncommutative setting.
In the general setting of Lie groupoids, this problem is still open, despite progress in particular cases
such as foliations and group actions. With this problem in mind, the first result of \cite{ppt} is the 
construction of a canonical map
\begin{equation}
\label{eq:char}
\chi:H^\bullet_{\rm diff}(\sfG;L_\g)\to HC^\bullet(\calA_\sfG)
\end{equation}
from the smooth groupoid cohomology with values in the bundle of transversal densities to the 
cyclic cohomology of $\calA_\sfG$. In \cite[\S 1]{ppt} an explicit formula on the cochain level is given,
but the actual origin of the map comes from the ``noncommutative symmetries'' of the algebra $\calA_
\sfG$, given in the form of an action of a Hopf algebroid. In any case, in degree zero, the map 
associates to an invariant transversal density $\Omega\in H^0_{\rm diff}(\sfG;L_\g)=\Gamma^\infty_{\rm inv}(M;L_\g)$ the trace
\[
\tau_\Omega(f):=\int_M f\Omega.
\]
Recall that the Lie groupoid is said to be {\em unimodular} if there exists a nowhere vanishing 
invariant transversal density $\Omega$. Therefore, for unimodular Lie groupoids we always have a 
trace on its convolution algebra. One can view the higher degree parts of the map (\ref{eq:char}) as 
a possible remedy for the absence of a trace in the non-unimodular case.

Next we introduce the {\em universal enveloping algebra} $\calU(\g)$ of the Lie algebroid $\g$. 
There is a completely algebraic construction of this algebra, but here we give a global description 
using the groupoid $\sfG$ integrating $\g$. For this we think of $\sfG$ as being fibered over
$M$ by the target map $t:\sfG\to M$. The tangent bundle along the fibers is then given by
$T_t\sfG:=\ker dt\cong s^*\g$. We now consider the algebra $\calD_{\rm inv}(\sfG)$ of 
families of differential
operators along the $t$-fibers that are $\sfG$-invariant: 
\[
D_{t(g)}U_g=U_g D_{s(g)},\quad D\in \calD_{\rm inv}(\sfG).
\]
Here $D_x,~x\in M$ means the restriction of $D$ to $t^{-1}(x)$, and 
$U_g:C^\infty(t^{-1}(s(g)))\to C^\infty(t^{-1}(t(g)))$ is defined by $(U_gf)(h):=f(hg^{-1})$.
Since such invariant differential operators are completely determined by their restriction to $M$,
their definition only depends on $\g$ and we also write $\calU(\g)$ for $\calD_{\rm inv}(\sfG)$.
This is the enveloping algebra. It carries an obvious filtration $\calU(\g)=\bigcup_{k\geq 0}U_k(\g)$
given by the order of a differential operator.

Obviously we can extend this definition to differential operators acting on equivariant vector bundles. 
Since such vector bundles are always of the form $t^*E$ for some vector bundle $E\to M$, we simply
write $\calU(\g;E)=\calU(\g)\otimes{\rm End}(E)$ for the resulting algebra. Using the fact that the 
fiberwise symbol of elements in $\calD_{\rm inv}(\sfG)$ commutes with the $\sfG$-action, 
the principal symbol map descends to the Lie algebroid level to
\[
\sigma_k:\calU_k(\g;E)\to \Gamma^\infty\left(M;{\rm Sym}^k\g\otimes{\rm End}(E)\right).
\]
We say that $D\in \calU_k(\g)$ is elliptic, if $\sigma_k(D)$ is invertible off the zero section in $\g^*$.
This means that the corresponding family $\tilde{D}\in \calD_{\rm inv}(\sfG)$ is elliptic on each 
fiber of the target map. 

In \cite{nwx}, an invariant pseudodifferential calculus $\Psi^\infty_{\rm inv}(\sfG)$ is described 
which forms a natural enlargement of $\calD_{\rm inv}(\sfG)$. It has the property that its ideal
of smoothing operators $\Psi^{-\infty}_{\rm inv}(\sfG)$ is isomorphic to the convolution algebra
$\calA_\sfG$. Ellipticity implies that $D$ is invertible in $\Psi^\infty_{\rm inv}(\sfG)$ modulo $\Psi^{-\infty}_{\rm inv}(\sfG)$, and therefore standard $K$-theoretic considerations yield an element
\[
{\rm ind}(D)\in K_0(\calA_\sfG),
\]
called the {\em index class}, in the $K$-theory of the convolution algebra.  

It is a fundamental fact in noncommutative geometry that there exists a canonical pairing between
the $K$-theory of an algebra and cyclic cohomology, c.f.\ \cite{connes}. In our case, this induces a
pairing between $H^{\rm ev}_{\rm diff}(\sfG;L_\g)$ and elliptic elements in $\calU(\g;E)$, called the 
{\em higher index}:
\[
{\rm ind}_\nu(D):=\left<\chi(\nu),{\rm ind}(D)\right>\in\C,\quad \nu\in H^{\rm ev}_{\rm diff}(\sfG;L_\g).
\]
When $\sfG$ is the pair groupoid $M\times M\rightrightarrows M$, and 
$\nu=1\in H^0_{\rm diff}(M\times M;L_{TM})$ is the unique trivial class in degree zero, ${\rm ind}_1(D)$ is just
the usual index of an elliptic differential operator. We are now in the position to state the index theorem
of \cite{ppt}. It gives an explicit expression for the higher index pairing as follows:
\begin{theorem}[\cite{ppt}]
\label{index-thm}
Let $\sfG\rightrightarrows M$ be a Lie groupoid over a compact base manifold $M$ with Lie algebroid $\g\to M$. Suppose that $D\in \calU(\g;E)$ is elliptic, where $E\to M$ is a vector bundle. Then, for $\nu\in H^{2k}_{\rm diff}(\sfG;L_\g)$, the following equality holds true;
\[
{\rm Ind}_\nu(D)=\frac{1}{(2\pi \sqrt{-1})^k}\int_{\g^*}\pi^*\Phi_\g(\nu)\wedge{\rm Td}^{\pi^!\g}(\pi^!\g\otimes\C)\wedge {\rm ch}^{\pi^!\g}(\sigma(D))
\] 
\end{theorem}
So far, we have explained the ingredients of the analytical left hand side of the equality in this index 
theorem.
 The topological right hand side of the index theorem is given in the framework of Lie algebroids and their cohomology, and its details will become clear in the next section. Let us already briefly summarize its main ingredients:
 \begin{itemize}
 \item $\pi^!\g$ denotes the pull-back Lie algebroid of $\g$ along the projection $\pi:\g^*\to M$. All
 the terms in the integrand on the right hand side are classes in the Lie algebroid cohomology of
 $\pi^!\g$. In \S \ref{integration} we explain how the integral is defined on this cohomology.
 \item $\Phi_\g:H^\bullet_{\rm diff}(\sfG;L_\g)\to H^\bullet(\g;L_\g)$ is the van Est map for the Lie 
 groupoid $\sfG$. Furthermore, $\pi^*:H^\bullet(\g;L_\g)\to H^\bullet(\pi^!\g;\pi^*L_\g)$ is the obvious 
 pull-back map, an isomorphism on the level of cohomology, c.f.\ \cite{crainic}.
 \item ${\rm Td}^{\pi^!\g}(\pi^!\g\otimes\C)$ and ${\rm ch}^{\pi^!\g}(\sigma(D))$ are characteristic classes in the Lie algebroid cohomology of $\pi^!\g$, called the Todd genus, c.f. equation \eqref{todd}, and the Chern character, c.f. equation \eqref{chern}.
 \end{itemize}
In \cite[\S 7]{ppt} it is discussed how to obtain the Atiyah--Singer index theorem and the index theorem 
for families and foliations from Theorem \ref{index-thm} above. The proof of Theorem \ref{index-thm} 
given in \cite{ppt} is by means of the algebraic index theorem for the Lie--Poisson structure on $\g^*$. 
In \cite{pptl2}, we have generalized Theorem \ref{index-thm} to study the index of an invariant elliptic 
operator for a proper cocompact $\mathsf{G}$-action. Applied to the canonical action of $\sfG$ on 
itself by left translation, we recover the index theorem above.

\section{The Thom isomorphism for Lie algebroids}
\label{Thom}
The aim of this section is to prove a version of the Thom isomorphism valid for general Lie algebroids.
In this context it is best to think of a Lie algebroid as a generalization of the tangent bundle of a 
manifold. Indeed, the Lie algebroid of the pair groupoid is equal to the tangent bundle, whose Lie 
algebroid cohomology is the usual de Rham cohomology of the underlying manifold. In this case 
we find back the usual Thom isomorphism for the cohomology of manifolds. 

\subsection{Representations and Lie algebroid cohomology}
As before, we let $\g\to M$ be Lie algebroid over a compact manifold $M$. Let $E\to M$ be a vector 
bundle (over $\R$ or $\C$). We define 
$\Omega^k_\g(E):=\Gamma^\infty(M;\bigwedge^k\g^*\otimes E)$, the space Lie algebroid $k$-forms
with values in $E$.
A {\em $\g$-connection} on $E$ is a linear operator
\begin{equation}
\label{connection}
\nabla:\Omega^0_\g(E)\to\Omega^1_\g(E),
\end{equation}
satisfying the Leibniz rule
\[
\nabla_X(fs)=f\nabla_X(s)+df(\rho(X))s,\quad X\in\Gamma^\infty(M;\g),~f\in C^\infty(M),~s\in\Gamma^\infty(M;E).
\]
It has a curvature tensor $R(\nabla)\in\Omega^2_\g({\rm End}(E))$ defined in the usual way by
\[
R(\nabla)(X,Y)=[\nabla_X,\nabla_Y]-\nabla_{[X,Y]}.
\]
When $R(\nabla)=0$, we say that $\nabla$ defines a {\em representation} of $\g$ on $E$. 
When $\sfG\rightrightarrows M$ is a Lie groupoid integrating $\g$, a representation of $\sfG$ on 
$E$ induces a representation of $\g$ on $E$. The trivial representation is defined on the trivial 
bundle $M\times\C\to M$ equipped with the $\g$-connection $\nabla_X(f):=\rho_\g(X)\cdot f$.

For a representation $(E,\nabla)$ of $\g$, the Lie algebroid cohomology complex is given by
$\Omega^\bullet_\g(E)$ with differential $d_\g:\Omega^k_\g(E)\to\Omega^{k+1}_\g(E)$ given by the usual Koszul formula:
\begin{equation}
\label{diff}
\begin{split}
d_\g\alpha(X_0,\ldots,X_{k}):=&\sum_{i=0}^k(-1)^i\nabla_{X_i}\alpha(X_0,\ldots,\hat{X}_i,\ldots,X_k)\\
&+
\sum_{i<j}(-1)^{i+j-1}\alpha([X_i,X_j],X_1,\ldots,\hat{X}_i,\ldots,\hat{X}_j,\ldots,X_{k}),
\end{split}
\end{equation}
where the hat means omission from the argument. 
When $\g=TM$, this definition agrees with the de Rham differential, and 
when $\mathfrak{g}$ is a Lie algebra, it reduces to the Chevalley--Eilenberg differential defining Lie 
algebra cohomology. We write $H^\bullet(\g;E)$ for its cohomology, for the trivial representation we 
simply write $H^\bullet(\g)$.

\subsection{Constructions of Lie algebroids}
For the Thom isomorphism we need several constructions with Lie algebroids, discussed in detail in 
\cite[\S 4]{mackenzie}. Here we briefly recall these constructions and indicate the corresponding 
algebraic structure on the cohomology induced by them.

\subsubsection*{Direct products of Lie algebroids}
Let $\g_1\to M_1$ and $\g_2\to M_2$ be two Lie algebroids over (possibly) different base manifolds. 
Consider the product manifold $M_1\times M_2$ equipped with the two projections $pr_1$ and $pr_2$ to $M_1$ resp. $M_2$. As shown in \cite[\S 4.2.]{mackenzie}, there is a unique Lie algebroid structure on the vector bundle $pr_1^*\g_1\oplus pr_2^*\g_2^*$ whose Lie bracket reduces to that of $\g_1$ and $\g_2$ for sections taking values in $\g_1$ and $\g_2$ that are constant along $M_2$ and $M_1$. Furthermore, the bracket between $\g_1$ and $\g_2$ is set to zero. On the level of Lie algebroid cohomology this leads to the external product map
\begin{equation}
\label{ext-prdt}
H^i(\g_1;E_1)\otimes H^j(\g_2;E_2)\to H^{i+j}(pr_1^*\g_1\oplus pr^*_2\g_2^*;E_1\boxtimes E_2),
\end{equation}
where $E_1$ is a representation of $\g_1$ and $E_2$ of $\g_2$.
\subsubsection*{Pull-back Lie algebroids}
Let $f:P\to M$ be a surjective submersion and $\g\to M$ a Lie algebroid over $M$. The pull-back Lie algebroid $f^!\g$ is defined on the vector bundle with typical fiber given at $p\in P$ by
\[
(f^!\g)_p=\{(X,\xi)\in \g_{f(p)}\oplus T_pP,~\rho_\g(X)=df(\xi)\}.
\]
The anchor is simply given by $\rho_{f^!\g}(X,\xi)=\xi$. The Lie bracket is defined by linear extension of 
the formula
\[
[(f_1X_1,\xi_1),(f_2X_2,\xi_2)]=(f_1f_2[X_1,X_2]+\xi_1(f_2)X_2-\xi_2(f_1)X_1,[\xi_1,\xi_2]),
\]
where $f_1,f_2\in C^\infty(P)$, $X_1,X_2\in \Gamma^\infty(M,\g)$ and $\xi_1,\xi_2\in\mathfrak{X}(P)$.

With this definition, for any representation $E$ of $\g$, its pull-back $f^*E$ carries a canonical representation of $f^!\g$. Furthermore, the pull-back along $f$ induces a morphism $f^*: \Omega_\g(E)\to\Omega^\bullet_{f^!\g}(f^*E)$ commuting with the differentials. Theorem 2 of \cite{crainic} states 
that the induced map
\[
f^*:H^\bullet(\g;E)\to H^\bullet(f^!\g;f^*E)
\]
is an isomorphism in degree $k\leq n$ when $f$ has homologically $n$-connected fibers. 

\subsubsection*{Morphisms of Lie algebroids} 
The notion of a morphism of Lie algebroids is a bit subtle if we have to change base: 
Let $\g_1\to M_1$ and $\g_2\to M_2$ be two Lie algebroids, and suppose that there exists a vector bundle morphism
\[
\xymatrix{\g_1\ar[r]^\varphi\ar[d]&\g_2\ar[d]\\ M_1\ar[r]^f&M_2}.
\]
Such a morphism of vector bundles induces a well-defined pull-back of sections: $(f,\varphi)^*:\Gamma^\infty(M_2;\g_2)\to\Gamma^\infty(M_1;\g_1)$.
We say that the pair $(f,\varphi)$ is a {\em morphism of Lie algebroids} if 
the following conditions hold true:
\begin{itemize}
\item[$i)$] (Compatibility with anchors)
\[
\rho_2\circ\varphi=Tf\circ\rho_1.
\]
\item[$ii)$] (Compatibility with brackets) If $X_1,Y_1\in\Gamma^\infty(M_1,\g_1)$ are mapped to $\varphi(X_1)=\sum_if_iX_{2,i}$ resp. $\varphi(Y_1)=\sum_jh_jY_{2,j}$, then 
\[
\varphi([X_1,Y_1])=\sum_{i,j}f_ih_j[X_{1,i},Y_{1,j}]+\sum_j \rho_1(X_1)(h_j) Y_{2,j}-\sum_i\rho_1(Y_1)(f_i)X_{2,i}.
\] 
\end{itemize}
These conditions are equivalent to requiring that the pull-back $(f,\varphi)^*:\Omega^\bullet_{\g_2}\to \Omega^\bullet_{\g_1}$ is a morphism of cochain complexes, i.e., commutes with the Chevalley--Eilenberg differential \eqref{diff}. More generally, for $E$ a representation of $\g_2$, one easily checks that $f^*E$, carries a canonical representation of $\g_1$. With this structure, the pull-back is defined
as a morphism $(f,\varphi)^*:\Omega^\bullet_{\g_2}(E)\to \Omega^\bullet_{\g_1}(f^*E)$ of cochain complexes.
On the level of cohomology we therefore obtain a map
\begin{equation}
\label{pull-back}
(f,\varphi)^*:H^\bullet(\g_2;E)\to H^\bullet(\g_1;f^*E).
\end{equation}
\subsection{The Lie algebroid Thom class}
We now come to the actual proof of the Thom isomorphism. In the following we denote by $Th^r\in H_c^r(\R^r)$ the canonical generator of the compactly supported cohomology of $\R^r$.
Let $\pi: E\to M$ be an orientable vector bundle of rank $r$. Choosing an orientation, we can define its {\em Thom class} $Th(E)\in H_c^r(E)$ uniquely by requiring the property that $Th(E)$ restricts consistently to $Th^r$ at each fiber $E_x\cong\R^r$ for all $x\in M$. The Thom class
defines an isomorphism
\[
H^k(M)\stackrel{\cong}{\longrightarrow} H^{k+r}_c(E),\quad \alpha\mapsto \pi^*\alpha\wedge Th(E),
\]
for all $\alpha\in H^k(M)$. Restriction to $M$, embedded via the zero section gives the Euler class $e(E)=Th(E)|_M$. In this section we will prove a similar statement for Lie algebroids.

Let, as before, $\g\to M$ be a Lie algebroid with anchor $\rho_\g:\g\to TM$, and $\pi:E\to M$ a vector bundle of rank $r$. 
\begin{definition}
The Lie algebroid Thom class is defined as
\begin{equation}
\label{def-thom}
Th^\g(E):=\rho_{\pi^!\g}^*Th(E)\in H^r_c(\pi^!\g).
\end{equation}
\end{definition}
The crucial properties of the Lie algebroid Thom class are summarized in the following theorem:
\begin{theorem}
\label{thom-iso}
Let $\g\to M$ (and $\g'\to M'$) be a Lie algebroid over a compact manifold $M$ (and $M'$) and $\pi:E\to M$ a smooth, oriented, rank $r$  vector bundle. The Thom class $Th^\g(E)\in H^r_c(\pi^!\g)$ is the unique cohomology class
with the following properties:
\begin{itemize}
\item[$i)$] For any morphism $(f,\varphi):\g'\to \g$ of Lie algebroids we have
\[
(f,\varphi)^*Th^\g(E)=Th^{\g'}(f^*E),
\]
\item[$ii)$] For $x\in M$, the inclusion map $i_x:E_x\to E$ extends to a morphism of Lie algebroids
\[
(i_x,\iota_x):(E_x,TE_x)\to (E,\pi^!\g)
\]
 under which 
 \[
 (i_x,\iota_x)^*Th^\g(E)=Th^r\in H^r_c(E_x)\cong H^r_c(\R^r).
 \]
 \item[$iii)$] The map $\Psi_E:H^\bullet(\g)\to H^{r+\bullet}_c(\pi^!\g)$ defined by
 \[
 \alpha\mapsto \pi^*\alpha\wedge Th^\g(E),
 \]
is an isomorphism $H^k(\g)\cong H^{k+r}_c(\pi^!\g)$. 
\end{itemize}
\end{theorem}
\begin{proof}
$i)$ The ordinary Thom class $Th(E)\in H^r_c(E)$ satisfies the naturality assumption
\[
Th(f^*E)=f^*Th(E),
\]
for a smooth map $f:M'\to M$. The first property therefore follows from the fact that for a morphism $(f,\varphi)$ of Lie algebroids, the pull-back \eqref{pull-back} fits into a commutative diagram
\begin{equation}
\label{useful}
\xymatrix{
H_c^\bullet(f^*E)\ar[r]^{f^*}\ar[d]_{\rho_2^*}&H_c^\bullet(E)\ar[d]^{\rho_1^*}
\\
H_c^\bullet(\pi^!\g')\ar[r]^{(f,\varphi)^*}&H_c^\bullet(\pi^!\g)}.
\end{equation}

$ii)$ For $x\in X$, we denote by $i_x:E_x\to E$ the inclusion of the fiber along $x$. Let $\xi\in T_eE_x$ be a tangent vector. Since $Ti_x(\xi)\in T_eE$ projects to zero along $T\pi$, we see that $(0,Ti_x(\xi))\in(\pi^!\g)_e$ which defines a map $\iota_x:TE_x\to\pi^!\g$. Since this map is basically just the derivative of the inclusion map $i_x$, it is easily seen that the pair $(i_x,\iota_x)$ forms a morphism of Lie algebroids, i.e., it is compatible with the commutators and the anchor maps. We can therefore pull-back the Lie algebroid Thom class. Since the Lie algebroid Thom class satisfies the property that $i_x^*Th(E)$ equals the generator $Th(E_x)$ of $H^r_c(E_x)$, property $ii)$ easily follows.

$iii)$ Let us first consider the case that $E\to M$ is trivial: $E\cong M\times\R^r$. In this case we have 
$\pi^!\g\cong pr_1^*\g\oplus pr_2^*T\R^r$ so that 
\[
H^\bullet(\g)\otimes H_c^\bullet(\R^r)\stackrel{\cong}{\longrightarrow}
H^\bullet_c(\pi^!\g),
\]
by means of the external product \eqref{ext-prdt}. The Thom class $Th^\g(M\times\R^r)\in H^r(\pi^!\g)$ corresponds on the left hand side to the generator $1\otimes Th^r$ of $H_c^\bullet(\R^r)$, so this proves the result for a trivial vector bundle.

The general case is proved by a Mayer--Vietoris argument, analogous to the proof of the ordinary Thom isomorphism in \cite[Thm 6.17]{bt}: For two open subsets $U$ and $V$ of $M$, the restriction maps give a sequence
\[
0\to\Omega^\bullet_{\pi^!\g,cv}(E|_{U\cup V})\to \Omega^\bullet_{\pi^!\g,cv}(E|_U)\oplus \Omega^\bullet_{\pi^!\g,cv}(E|_{V})\to \Omega^\bullet_{\pi^!\g,cv}(E|_{U\cap V})\to 0,
\]
where $cv$ stands for compact support along the fibers of $E$. As for the ordinary Mayer--Vietoris sequence, c.f.\ \cite[\S 2]{bt}, a partition of unity subordinate to $U,V$ allows to prove that this sequence is exact. We apply the Thom map $\Psi_E$ to the induced  long exact sequence  in cohomology to find:
\[
\minCDarrowwidth15pt
\begin{CD}
\ldots@>>>
H^\bullet_{cv}(\g|_{U})\oplus H^\bullet_{cv}(\g|_{V})@>>>H^\bullet_{cv}(\g|_{U\cap V})@>\delta>> H^{\bullet+1}_{cv}(\g|_{U\cup V})@>>>\ldots\\
@. 
@VV\Psi_{E|_{U}}\oplus \Psi_{E|_{V}}V @VV\Psi_{E|_{U\cap V}}V @VV\Psi_{E|_{U\cup V}}V @. \\
\ldots @>>>
H^{\bullet+r}_{cv}(\pi^!\g|_{U})\oplus H^{\bullet+r}_{cv}(\pi^!\g|_{V})@>>> H^{\bullet+r}_{cv}(\pi^!\g|_{{U\cap V}}) @>\delta>> H^{\bullet+r+1}_{cv}(\pi^!\g|_{{U\cup V}}) @>>>\ldots
\end{CD}
\]
Here $\pi^!\g|_U$ stands short for $\pi^!\g|_{E|_U}$, etc.  It is easy to check that each square in this diagram commutes. Now suppose that $U$ and $V$ are such that $E|_U$ and $E|_V$ are trivial, and $U\cap V$ contractible. Then $E|_{U\cap V}$ is trivial as well, so that $\Psi_{E|_U}$, $\Psi_{E|_V}$ and 
$\Psi_{E|_{U\cap V}}$ are isomorphisms by the previous part of the proof above. By the five-Lemma, $\Psi|_{E|_{U\cup V}}$ must be an isomorphism as well. By choosing a good cover of $M$, the general result is proved by induction on the cardinality of the cover.
\end{proof}

The map $i:M\to E$ embedding $M$ by means of the zero section fits into a diagram
\[
\xymatrix{\g\ar[r]^\iota\ar[d]&\pi^!\g\ar[d]\\ M\ar[r] &E}
\]
where $\iota(X)=(X,\rho(X))\in(\pi^!E)_{i(x)}$, for all $X\in \g_x, x\in M$. 
We can therefore pull-back the Thom class to get the {\em Lie algebroid} Euler class
\begin{equation}
\label{euler}
e^\g(E):=(i,\iota)^*Th^\g(E)\in H^r(\g).
\end{equation}
In fact, it is not difficult to verify that $e^\g(E)=\rho_\g^*e(E)$, where $e(E)\in H^r(M)$ is the usual Euler characteristic of the vector bundle $E$ in the ordinary cohomology of $M$.

\subsection{Integration}
\label{integration}
Recall that $\sfG\rightrightarrows M$ is said to be {\em unimodular} if there exists a non vanishing, invariant ``volume form'' 
$\Omega\in\Gamma^\infty_{\rm inv}\left(M;L_\g\right)$, where $L_\g$ is the bundle of densities \eqref{density-bundle}.  With such an invariant volume form, we can define an integration map on Lie algebroid cohomology as follows: for $\alpha\in \Omega^{\rm top}_\g$, we pair with $\Omega$ to obtain a top degree form on $M$ that we can integrate:
\[
\alpha\mapsto \int_M\left<\alpha,\Omega\right>.
\]
Recall that $M$ is assumed to be compact, otherwise we have to assume the Lie algebroid cochains
to have compact support on $M$ in the following discussion. 
\begin{lemma}
The integration map above descends to cohomology, i.e.,
\[
\int_M\left<d_\g\beta,\Omega\right>=0,\quad \mbox{for all}~\beta\in \Omega^{r-1}_\g.
\]
\end{lemma}
\begin{proof}
This is proved in \cite[\S 5]{elw}.
\end{proof}
Next, we consider the compatibility of the Thom isomorphism with integration. For this, 
we  consider the case $E=\g^*$, the dual of the Lie algebroid itself. As remarked in \cite{ppt}, the pull-back Lie algebroid $\pi^!\g$ has a canonical symplectic structure
$\Theta\in\Omega^2_{\pi^!\g}$ defined by 
\begin{equation}
\label{sympl-form}
\Theta(X,Y):=\left<p_1(X),p_2(Y)\right>-\left<p_1(Y),p_2(X)\right>,
\end{equation}
where $\left<~,~\right>$ means the canonical pairing between $\g$ and $\g^*$ and  $p_1,p_2$ fit into the commutative diagram
\[
\xymatrix{\pi^!\g\ar[r]^{p_1}\ar[d]_{p_2}&\g^*\ar[d]^\pi\\\g\ar[r]&M}.
\]
We now consider the element $\Omega_{\pi^!\g}:=\Theta^r\otimes\pi^*\Omega$. Since $(\pi^!\g)_\xi=\g_{\pi(\xi)}\oplus \g^*_{\pi(\xi)}$ for all $\xi\in \g^*$, we have
\begin{equation}
\label{volume-pull-back}
(\Theta^r\otimes\pi^*\Omega)_\xi\in \bigwedge^{\rm top} \g_{\pi(\xi)}^*\otimes\bigwedge^{\rm top}\g_{\pi(\xi)}\otimes\bigwedge^{\rm top}\g_{\pi(\xi)}\otimes \bigwedge^{\rm top}T_{\pi(\xi)}^*M
\end{equation}
and we can view $\Theta^r\otimes\pi^*\Omega$ as a ``volume form'' on $\pi^!\g$. According to  
\cite[Lemma 4.9]{ppt}, this volume form is invariant and therefore defines an integration map
\[
\int_{\g^*}\left<-,\Omega_{\pi^!\g}\right>:H^{\rm top}_c(\pi^!\g)\to\C.
\]
\begin{proposition}
\label{Thom-int}
Let $\alpha\in H^r(\g)$. Then the following equality holds true:
\[
\int_M\left<\alpha,\Omega\right>=\int_{\g^*}\left<\pi^*\alpha\wedge Th^\g(\g^*),\Omega_{\pi^!\g}\right>.
\]
\end{proposition}
\begin{proof}
We have seen in the proof of Theorem \ref{thom-iso} that in a local trivialization $\g^*|_U\cong U\times \R^r$, 
\[
\pi^!\g|_U\cong pr_1^*\g\oplus pr_2^*T\R^r,
\]
and that the Thom class $Th^\g(\g^*)$ restricts to the generator $Th^r\in H^r_c(\R^r)$. Since the volume form \eqref{volume-pull-back} is just the top exterior power of the canonical symplectic form on $pr_2^*T^*\R^2$, it follows that 
\[
\begin{split}
\left<\pi^*\alpha\wedge Th^\g(\g^*),\Omega_{\pi^!\g}\right>&=\left<\pi^*\alpha,\pi^*\Omega\right>\otimes Th^r\\
&=\pi^*\left<\alpha,\Omega\right>\otimes Th^r\in H^{\rm top}_{cv}(\g^*|_U)
\end{split}
\]
To perform the integral over $\g^*$, we first integrate over the fibers of the vector bundle $\g^*\to M$. Now, because
\[
\int_{\R^r} Th^r=1,
\]
the result is an integral over $M$, exactly equal to the left hand side of the asserted equality in the statement of the proposition. This completes the proof.
\end{proof}
When $\sfG$ is not unimodular, the density bundle should be included in the Lie algebroid cohomology
classes that one wants to integrate. With this the arguments above generalize easily and we state the 
final result without proof:
\begin{theorem}
Let $\g\to M$ be a Lie algebroid of rank $r$. For $\alpha\in \Omega^r_\g(L_\g)$, the integral
\[
\int_M\alpha
\]
is well-defined and depends only on the cohomology class in $H^r(\g;L_\g)$ determined by $\alpha$. 
The Thom class $Th^{\g}(\g^*)\in H^r_c(\pi^!\g)$ yields an isomorphism
\[
\Psi_{\g^*}:H^\bullet(\g;L_\g)\stackrel{\cong}{\longrightarrow} H^{\bullet+r}_c(\pi^!\g;L_{\pi^!\g}),
\quad \alpha \mapsto (\pi^*\alpha\otimes\Theta^r)\wedge Th^{\g}(\g^*),
\]
compatible with integration:
\[
\int_M\alpha=\int_{\g^*}\Psi_{\g^*}(\alpha).
\]
\end{theorem}

\subsection{Characteristic classes of Lie algebroids}
\label{char}
As usual, $\g\to M$ is a Lie algebroid with anchor $\rho_\g$.
The Euler class \eqref{euler} of a vector bundle introduced above is an example of a so-called {\em primary characteristic class} in Lie algebroid cohomology. These are characteristic classes of vector bundles $E\to M$ obtained by pull-back along the anchor map $\rho_\g:\g\to TM$ of the usual characteristic classes in cohomology. In the special case of a foliation, these characteristic classes 
are also known as {\em foliated characteristic classes}, described in \cite[Ch. V]{ms}. An alternative 
approach, c.f. \cite{crainic,fernandes}, to these primary characteristic classes is to apply the usual 
Chern--Weil theory to a {\em $\g$-connection} $\nabla$ on $E$ as in \eqref{connection}: such a 
connection has a curvature tensor $R(\nabla)\in\Omega^2_\g({\rm End}(E))$ and the main theorem of 
Chern--Weil then asserts that for any degree $k$ polynomial $P:M_{\rk(E)}(\C)\to\C$ invariant under 
the conjugacy action of $GL(\rk(E),\C)$, the Lie algebroid cochain
\[
P(R(\nabla))\in \Omega^{2k}_\g
\]
is well-defined, closed and its cohomology class in $H^{2k}(\g)$ is independent of the connection 
chosen. The equality of the two approaches, i.e., pulling back along the anchor and Chern--Weil for 
$\g$-connections, follows from the fact a connection on $E$ induces a $\g$-connection by the anchor map.

The most basic characteristic classes of complex vector bundles are of course the Lie algebroid {\em Chern classes} $c_i^\g(E)\in H^{2i}(\g)$ defined as
\[
\det(1+R(\nabla))=1+c^\g_1(\nabla)+c^\g_2(\nabla)+\ldots
\]
The {\em Chern character} is given by
\begin{equation}
\label{chern}
\begin{split}
{\rm ch}^\g(\nabla):&={\rm Tr}\left(e^{R(\nabla)}\right)
={\rm rk}(E)+c^\g_1(\nabla)+\frac{1}{2}(c^\g_1(\nabla)^2-2c^\g_2(\nabla))+\ldots
\end{split}
\end{equation}
Another example, appearing in the index theorem \ref{index-thm} is the {\em Todd class}:
\begin{equation}
\label{todd}
{\rm Td}^\g(\nabla):=\det\left(\frac{R(\nabla)}{1-e^{R(\nabla)}}\right)= 1+\frac{1}{2}c_1^\g(\nabla)+
\frac{1}{12}(c_2^\g(\nabla)+ c_1^\g(\nabla)^2)+\ldots
\end{equation}

\section{Examples of geometric operators}
\label{geom}
We now return to the index theorem on Lie groupoids. Using the theory of the previous section we can 
use the Thom isomorphism $\Psi_{\g^*}$ to rewrite the right hand side of the index Theorem \ref{index-thm} as an integral over $M$ instead of $\g^*$: for this we have to compute $\Psi_{\g^*}^{-1}$ of the Lie algebroid cohomology class in $H^\bullet_c(\pi^!\g)$ that appears in the integral on the right hand side. For this we
use Proposition \ref{Thom-int} which we rewrite as
\[
\int_{\g^*}\gamma=\int_M\Psi_{\g^*}^{-1}(\gamma),\quad \gamma\in H^{\rm top}_c(\pi^!\g).
\]
Now if $\gamma=\Psi_{\g^*}(\nu)=\pi^*\nu\wedge Th^\g(\g)$, for $\gamma\in H^{\bullet}_c(\pi^!\g)$ and $\nu\in H^\bullet(\g)$,  we have that 
\[
(i,\iota)^*\gamma=\nu\wedge e^\g(\g).
\]
Therefore $\nu=(i,\iota)^*\gamma\slash e^\g(\g)$, provided we can perform the division. For the index theorem, we thus find, using the fact that  $(i,\iota)^*{\rm ch}^{\pi^!\g}(\sigma(D))={\rm ch}^\g(E)$:
\begin{equation}
\label{index-alt}
{\rm ind}_{\alpha}(D)=\frac{1}{(2\pi \sqrt{-1})^k}\int_{M}\frac{\alpha\wedge{\rm Td}^\g(\g\otimes\C)\wedge {\rm ch}^\g(E)}{e^\g(\g)},
\end{equation}
provided we can make sense of the division by the Euler class. 
In this section we describe several natural geometric elliptic operators on Lie groupoids, following \cite{ln}, 
to which our result applies. We work out the index theorem for the Euler, signature and Dirac operator and obtain 
integral formulae for their indices, completely analogous to the topological index of these 
operators on compact manifolds, where de Rham cohomology has been replaced by 
the Lie algebroid cohomology of $\g$.
\subsection{The Euler operator}
The most natural differential operator on a manifold is of course the de Rham operator acting on sections of exterior powers of the dual of the tangent bundle. Its analogue for Lie algebroid is the Chevalley--Eilenberg operator \eqref{diff}, which can be viewed as an element
\[
d_\g\in\mathcal{U}\left(\g; \bigwedge \g^*\right).
\] 
As is the case for the de Rham operator, this Chevalley--Eilenberg operator is not elliptic. To obtain an elliptic operator, we fix a metric on $\g$, which defines, in the usual manner, a Hodge $*$-operator 
\[
*:\Omega^k_\g\to\Omega^{r-k}_\g.
\]
With the Hodge $*$, we can define the formal adjoint $\delta_\g:\Omega^{k+1}_\g\to\Omega^k_\g$ by the usual formula
\[
\delta_\g:=(-1)^{rk+r+1}*d_\g*.
\]
The Euler operator 
\[
D_\g:=d_\g+\delta_\g:\Omega^{ev}_\g\to\Omega^{odd}_\g
\]
is elliptic and the index Theorem \ref{index-thm} applies. To state the final result, we need the Euler 
class $e^\g(\g)\in H^\bullet(\g)$ introduced in equation \eqref{euler}. This class can be represented 
explicitly by a differential form using the {\em Pfaffian} of the curvature of the induced 
Levi--Civita connection $\nabla$, which is $\mathfrak{so}(r)$-valued:
\[
e^\g(\g)={\rm Pf}(R(\nabla))\in H^{2r}(\g).
\]
The Pfaffian is the invariant polynomial on $\mathfrak{so}(r)$ of degree $r$ satisfying ${\rm Pf}(A)^2=\det(A)$, so by Chern--Weil theory, the resulting cohomology class is independent of the metric chosen.
The final result is the following version of the Chern--Gauss--Bonnet theorem for Lie groupoids:
\begin{theorem}
\label{Gauss-Bonnet}
Let $\sfG\rightrightarrows M$ be a unimodular  Lie groupoid, equipped with a metric on its Lie algebroid. The index of the Euler operator $D_\g$ is given by
\[
{\rm Ind}_\Omega(D_\g)=\int_M \left<e^\g(\g),\Omega\right>.
\]
\end{theorem}
\begin{proof}
Using \eqref{index-alt}, the statement follows from the equality
\[
\frac{{\rm ch}^\g\left(\sum_{i\geq 0}(-1)^i\bigwedge^i\g^*\otimes\C\right){\rm Td}^\g(\g\otimes\C)}{e^\g(\g)}=e^\g(\g),
\]
which holds in $H^{2r}(\g)$. In ordinary cohomology, i.e., when $\g=TM$, this equality is well-known, 
see e.g. \cite[\S 4]{shanahan}. But since all the characteristic classes are primary, and obtained by 
pull-back along the anchor of the ordinary characteristic classes, the commutative diagram 
\eqref{useful} shows that the identity above holds true in Lie algebroid cohomology as well.
This proves the theorem.
\end{proof} 

Next, choose a generic section $X$ of $\g$, and denote the zero set of $X$ in $M$ by $Z$. By translation, the section $X$ lifts to a $\sfG$-invariant section $\widetilde{X}$ of $T_t\sfG$, the fiberwise tangent bundle of $t:\sfG\to M$. Denote by $\widetilde{Z}$ to be the zero set of $\widetilde{X}$, which is a $\sfG$ invariant and $\sfG$-cocompact subset of $\sfG$.

Let $e:=t|_{\widetilde{Z}}:\widetilde{Z}\to M$ be the restriction map. The differential of the section map $\widetilde{X}:\sfG\to T_t\sfG$ induces a $\sfG$-equivariant isomorphism from the normal bundle $N$ of $\widetilde{Z}$ in $\sfG$ and $T_t\sfG$, the fiberwise tangent bundle. By the $\sfG$-equivariant isomorphisms
\[
N\oplus T\widetilde{Z}=T\sfG|_{\widetilde{Z}}=(\tau\oplus T_t\sfG)|_{\widetilde{Z}}
\]
for the transverse bundle $\tau$ of $T_t\sfG$ in $T\sfG$, we can equip the map $e$ with a $\sfG$-equivariant $K$-orientation, and therefore $e_!(\widetilde{Z})$ defines, after composing with the ``quantization map'' $K(\g^*)\to K(C^*_r(\sfG))$,  an element in  $K(C_r^*(\sfG))$.  

\begin{proposition}
The index of  the Euler operator $D_\g$ is equal to $e_!(\widetilde{Z}) \in K(C^*_r(\sfG))$. 
\end{proposition}

\begin{proof}This is a line by line $\sfG$-equivariant generalization of the Connes-Skandalis proof of \cite[Corollary 4.17]{cs}. 
\end{proof}

When $\sfG$ is unimodular, the trace $\tau_\Omega$ on $\mathcal {A}_{\sfG}$ extends to a trace on $C^*_r(\sfG)$, and therefore we obtain the following formula
\[
{\rm Ind}_\Omega(D_\g)=\int_M \frac{{\rm ch}^\g(e_!(\widetilde{Z})){\rm Td}^\g(\g\otimes\C)}{e^\g(\g)},
\]
under the assumption that the division in the integrand on the right hand side can be made sense of. 
A careful understanding of the right hand side should lead to a generalization of the Poincar\'e--Hopf 
index theorem for groupoids, c.f. \cite{connes-foliations,be}.

\subsection{The signature operator}
We now assume that the Lie algebroid $\g\to M$ of $\sfG$ has even rank: $r=2p$.
In this case there exists an involution $\tau$ on $\bigoplus_{i\geq 0}\bigwedge^i \g^*$ given by
\[
\tau=(-1)^p*.
\]
It satisfies $\tau^2=1$ and induces a decomposition $\Omega_\g=\Omega^+_\g\oplus\Omega^-_\g$. 
Furthermore, $\tau D_\g=-D_\g\tau$, so that we can define the signature operator
\[
D_\g^{sign}:\Omega_\g^+\to\Omega_\g^-.
\]
Just as for the Euler operator, we can view the signature operator as an elliptic element 
\[
D_\g^{sign}\in\calU\left(\g; \bigwedge\g^*\right),
\]
so that the index Theorem \ref{index-thm} applies. To state the result we need the analogue of the {\em $L$-genus} in Lie algebroid cohomology:
\[
L^\g(\nabla):=\det\left(\frac{R(\nabla)}{\tanh(R(\nabla))}\right)=1+\frac{1}{3}p_1^\g(\nabla)+\frac{1}{45}(7p_2^\g(\nabla)-p_1^\g(\nabla)^2)+\ldots
\]
\begin{theorem}
Let $\sfG\rightrightarrows M$ be a Lie groupoid equipped with a metric on its Lie algebroid. Then, for $\nu\in H^{2k}(\g;L_\g)$, the
higher index of the signature operator is given by
\begin{equation}
\label{eq:sign}
{\rm Ind}_\nu(D^{sign}_\g)=\frac{1}{(2\pi\sqrt{-1})^k}\int_M\Phi_\g(\nu)\wedge L^\g(\g).
\end{equation} 
\end{theorem}
\begin{proof}
The strategy of the proof is the same as that of Theorem \ref{Gauss-Bonnet}, this time the relevant 
identity is:
\[
\frac{{\rm ch}^\g\left(\bigwedge^+\g^*-\bigwedge^-\g^*\right){\rm Td}^\g(\g\otimes\C)}{e^\g(\g)}=L^\g(\g),
\]
valid in $H^{ev}(\g)$. For the proof of this identity in de Rham cohomology, see \cite[\S 6]{shanahan}. The general case follows from \eqref{useful}.
\end{proof} 
For the monodromy groupoid of a foliation, and coupled to a foliated vector bundle, the index formula on the right hand side of equation (\ref{eq:sign}) gives some of the ``higher signatures'' of the foliation, which are the central objects of study in the Novikov conjecture for foliations, c.f.\ \cite{bc,hurder,bh}. In general,  the right hand side of equation (\ref{eq:sign}) defines higher signatures for a general Lie groupoid. The homotopy invariance property of these numbers will be studied in a future publication.

\subsection{The Dirac operator}
Associated to a metric $\left<~,~\right>$ on $\g\to M$ there is a unique torsion free $\g$-connection on $\g$ satisfying
\[
L_{\rho(X)}\left<Y,Z\right>=\left<\nabla_XY,Z\right>+\left<Y,\nabla_XZ\right>,\quad\mbox{for all}~X,Y,Z\in\Gamma^\infty(M;\g).
\]
It is called the Levi--Civita connection associated to the metric. Alternatively, by pull-back, the metric defines a $\sfG$-invariant Riemannian metric on $T_t\sfG$ along the $t$-fibers of $
\sfG$. This Riemannian metric induces a fiberwise Levi--Civita connection $\nabla$ which identifies with the pull-back of $\nabla$. Therefore, we can view $\nabla_X$ as an element in $\calU(\g;\g)$.

Let ${\rm Cliff}(\g)$ be the Clifford algebra bundle over $M$ associated to the pair $(\g,\left<~,~\right>)$. 
Let $S$ be a bundle of modules over ${\rm Cliff}(\g)$, with Clifford multiplication written as $s\mapsto c(X)s$, for $s\in\Gamma^\infty(M;S)$ and $X\in \Gamma^\infty(M;\g)$.
Assume that $S$ is equipped with a metric $\left<~,~\right>$ and a compatible $\g$-connection $\nabla^S$ such that 
\begin{itemize}
\item[$i)$] Clifford multiplication is skew-symmetric, i.e., 
\[
\left<c(X)s_1,s_2\right>+\left<s_1,c(X)s_2\right>=0,~\quad X\in \Gamma^\infty(M;\g),~s_1,s_2\in\Gamma^\infty(M;S),
\]
\item[$ii)$] the connection is compatible with the Levi--Civita connection:
\[
\nabla^S_X(c(Y)s)=c(\nabla_XY)s+c(Y)\nabla^S_Xs,\quad X,Y\in\Gamma^\infty(M;\g),~s\in\Gamma^\infty(M;S).
\]
\end{itemize}
The {\em Dirac operator} associated to these data is defined as
\[
\slashed{D}_S:=c\circ\nabla^S:\Gamma^\infty(M;S)\to \Gamma^\infty(M;S),
\]
where we have used the metric to identify $\g\cong \g^*$. In a local orthonormal frame $e_\alpha,~\alpha=1,\ldots,r$ of $\g$, this operator is written as
\begin{equation}
\label{local}
\slashed{D}_Ss=\sum_{\alpha=1}^rc(e_\alpha)\nabla_\alpha s,\quad s\in\Gamma^\infty(M;S).
\end{equation}
With this local expression, one easily proves the analogue of the Weitzenbock formula:
\begin{equation}\label{eq:wei}
\slashed{D}^2_S=(\nabla^S)^*\nabla^S+\sum_{\alpha<\beta}c(e_\alpha) c(e_\beta) R(\nabla^S)_{\alpha\beta},
\end{equation}
where $R(\nabla^S)_{\alpha\beta}$ denote the components of the $\g$-curvature of $\nabla^S$.
For the spinor bundle $\mathscr{S}$, defined below, this term is exactly $1/4$ times the scalar curvature
of $\nabla^\mathscr{S}$, and the analogue of the formula above is known as Lichnerowicz formula.

We now assume that the Lie algebroid $\g$ is {\em spin}: it is orientable and its structure group $SO(r)$ 
can be lifted to the double cover $Spin(r)$. When the rank $r$  of $\g$ is even, there is a unique Clifford
module $\mathscr{S}$, called the {\em spinor bundle}, whose fiber $\mathscr{S}_x$ at each point $x\in M$ equals the unique 
irreducible representation of ${\rm Cliff}(\g_x)$. It carries a unique metric compatible with the Levi--Civita  
connection satisfying the properties above, and we write $\slashed{D}$ for the associated Dirac 
operator. Given an arbitrary vector bundle $E$ over $M$ with connection $\nabla^E$, we can consider 
the ``twisted Dirac operator'' $\slashed{D}_E$ by using the Clifford bundle $E\otimes\mathscr{S}$ 
equipped with the connection 
$\nabla^{E\otimes\mathscr{S}}=\nabla^E\otimes 1+1\otimes\nabla^\mathscr{S}$.

As for the previous geometric operators, by pulling back to $\sfG$ using the target map, we can 
interpret this Dirac operator as an element
\[
\slashed{D}_E\in\mathcal{U}(\g;E\otimes\mathscr{S}).
\]
From the local expression \eqref{local} we read off that its symbol is given by
\[
\sigma_\xi(\slashed{D}_E)=i\xi,\quad \xi\in\g^*\cong\g,
\]
and therefore the operator is elliptic. Since the rank $r$ is even, the spinor bundle splits into a direct 
sum $\mathscr{S}=\mathscr{S}^+\oplus\mathscr{S}^-$, and the Dirac operator decomposes as
\[
\slashed{D}_E=\left(\begin{matrix}0&\slashed{D}^-_E\\ \slashed{D}_E^+&0\end{matrix}\right).
\]
Since $\slashed{D}_E$ is selfadjoint, its index class ${\rm Ind}(\slashed{D}_E)\in K_0(\calA_\sfG)$ is 
zero. However, using $\slashed{D}^+_E$ instead, we get an interesting class, and can apply Theorem
\ref{index-thm}. This leads to the following:
\begin{theorem} 
Let $\sfG\rightrightarrows M$ be a Lie algebroid with Lie algebroid $\g\to M$ of even rank which is spin. 
Assume $E\to M$ is vector bundle and $\nu\in H^{2k}_{\rm diff}(\sfG;L_\g)$. Then the following identity holds true;
\[
{\rm ind}_\nu(\slashed{D}^+_E)=\int_M\Phi_\g(\nu)\wedge\hat{A}^\g(\g)\wedge{\rm ch}^\g(E).
\]
\end{theorem}
In the theorem above, the Lie algebroid $\hat{A}^\g$-class is defined as usual by
\[
\hat{A}^\g(\nabla):=\det\left(\frac{R(\nabla)\slash 2}{\sinh(R(\nabla)\slash 2)}\right)=1-\frac{1}{24}p_1^\g(\nabla)+\frac{1}{5760}(7p_1^\g(\nabla)^2-4p_2^\g(\nabla)) +\ldots
\]
For varying vector bundles, the map $E\mapsto {\rm ind}_\nu(\slashed{D}^+_E)$ is easily seen to
descend to an index map on topological $K$-theory;
\[
K^0_{\rm top}(M)\to\C.
\]
There is a canonical forgetful map $R(\sfG)\to K^0_{\rm top}(M)$, where $R(\sfG)$ is the representation ring
of $\sfG$, the Grothendieck group completion of the abelian semigroup of isomorphism classes of
representations of $\sfG$. The following corollary shows that the index map only detects the rank
of the representation:
\begin{corollary}\label{cor:dirac}
Let $E\in {\rm Rep}(G)$, and $\nu\in H^{2k}_{\rm diff}(\sfG;L)$. Then 
\[
{\rm ind}_\nu(\slashed{D}_E)=\frac{{\rm rk}(E)}{(2\pi\sqrt{-1})^k}\int_M\Phi(\nu)\wedge\hat{A}^\g(\g).
\]
\end{corollary}
\begin{proof}
As observed in \cite{crainic}, for $E\in {\rm Rep}(\sfG)$, the Chern character in Lie algebroid 
cohomology is trivial: ${\rm ch}^\g_{2k}(E)=0$ in $H^{2k}(\g)$, for $k\geq 1$. This is immediately 
clear from the Chern--Weil construction of \S \ref{char}.
Therefore, ${\rm ch}^\g(E)={\rm rk}(E)\in H^0(\g)$ and the result follows.
\end{proof}

Finally, let us assume that $M$ is compact and that the scalar curvature of $\nabla^\mathscr{S}$ is positive at every point of $M$. We observe from the Weitzenb\"ock--Lichnerowicz formula (\ref{eq:wei}) that under these assumptions, the operator $\slashed{D}^2_\mathscr{S}$ is positive and therefore $\slashed{D}_\mathscr{S}$ has trivial kernel and cokernel over every $t$-fiber of $t:\sfG\to M$. As $M$ is compact, a standard operator algebra argument \cite{connes:fund} shows that the index ${\rm ind}(\slashed{D}_\mathscr{S})\in K_0(C^*_r(\sfG))$ is zero. When $\sfG$ is furthermore unimodular, the triviality of ${\rm ind}(\slashed {D}_\mathscr{S})\in K(C^*_r(\sfG))$ implies that 
\[
\tau_\Omega({\rm ind}(\slashed{D}_\mathscr{S}))=0. 
\]
Combing this with Corollary \ref{cor:dirac}, we conclude that 
\begin{corollary}
Assume $M$ is compact and $\sfG\rightrightarrows M$ is unimodular. If the scalar curvature of $\nabla^\mathscr{S}$ is  positive at every point of $M$, 
\[
\int_M\left<\hat{A}^\g(\g),\Omega\right>=0. 
\]
\end{corollary}

\end{document}